\documentclass{amsart}
\usepackage{amssymb}
\usepackage{amsmath}
\usepackage{amsfonts}

\newtheorem{theorem}{Theorem}
\theoremstyle{plain}

\newtheorem{corollary}{Corollary}

\newtheorem{lemma}{Lemma}

\numberwithin{equation}{section}
\input{tcilatex}

\begin{document}
\title[On BAI and Existence of Dense Ideals in Real and Jordan Algebras]{On
Bounded Approximate Identities and Existence of Dense Ideals in Real Locally
C*- and Locally JB-Algebras}
\author{Alexander A. Katz}
\address{Dr. Alexander A. Katz, Department of Mathematics and Computer
Science, St. John's College of Liberal Arts and Sciences, St. John's
University, 300 Howard Avenue, DaSilva Academic Center 314, Staten Island,
NY 10301, USA}
\email{katza@stjohns.edu}
\author{Oleg Friedman}
\address{Oleg Friedman, Department of Mathematical Sciences, University of
South Africa, P.O. Box 392, Pretoria 0003, South Africa}
\email{friedman001@yahoo.com}
\curraddr{Oleg Friedman, Department of Mathematics and Computer Science, St.
John's College of Liberal Arts and Sciences, St. John's University, 8000
Utopia Parkway, St. John's Hall 334, Queens, NY 11439, USA}
\email{fridmano@stjohns.edu}
\thanks{The second author is thankful to the first author and Dr. Louis E.
Labuschagne (University of South Africa, Pretoria, South Africa) for
constant help, support and discussions.}
\date{September 1, 2008}

\begin{abstract}
It has been established by Inoue that a complex locally C*-algebra with a
dense ideal posesses a bounded approximate identity which belonges to that
ideal. It has been shown by Fritzsche that if a unital complex locally
C*-algebra has an unbounded element then it also has a dense one-sided
ideal. In the present paper we obtain analogues of the aforementioned
results of Inoue and Fritzsche for real locally C*-algebras (projective
limits of projective families of real C*-algebras), and for locally
JB-algebras (projective limits of projective families of JB-algebras).
\end{abstract}

\maketitle

\section{Introduction}

Banach associative regular *-algebras over $%
\mathbb{C}
$, so called \textit{C*-algebras}, were first introduces in 1940's by
Gelfand and Naimark in the paper \cite{GelfandNaimark43}. Since then these
algebras were studied extensively by various authors, and now, the theory of
C*-algebras is a big part of Functional Analysis with applications in almost
all branches of Modern Mathematics and Theoretical Physics. For the basics
of the theory of C*-algebras, see for example Pedersen's monograph \cite%
{Pedersen79}.

The real analogues of complex C*-algebras, so called \textit{real C*-algebras%
}, which are real Banach *-algebras with regular norms such that their
complexifications are complex\textit{\ C*-algebras, }were studied in
parallel by many authors. For the current state of the basic theory of real
C*-algebras, see Li's monograph \cite{Li03}.

The real Jordan analogues of complex C*-algebras, so called \textit{%
JB-algebras}, were first defined by Alfsen, Schultz and St\o rmer in \cite%
{AlfsenSchultzStoermer78} as the real Banach--Jordan algebras satisfying for
all pairs of elements $x$ and $y$ the inequality of fineness%
\begin{equation*}
\left\Vert x^{2}+y^{2}\right\Vert \geq \left\Vert x\right\Vert ^{2},
\end{equation*}%
and regularity identity 
\begin{equation*}
\left\Vert x^{2}\right\Vert =\left\Vert x\right\Vert ^{2}.
\end{equation*}%
The basic theory of JB-algebras is fully treated in monograph of
Hanche-Olsen and St\o rmer \cite{Hanche-OlsenStoermer84}. If $A$ is a
C*--algebra, or a real C*-algebra, then the self-adjoint part $A_{sa}$ of $A$
is a JB-algebra under the Jordan product 
\begin{equation*}
x\circ y=\frac{(xy+yx)}{2}.
\end{equation*}%
Closed subalgebras of $A_{sa}$, for some C*-algebra or real C*-algebra $A$,
become relevant examples of JB-algebras, and are called \textit{JC-algebras}.

Complete locally multiplicatively-convex algebras or equivalently, due to
Arens-Michael Theorem, projective limits of projective families of Banach
algebras, were first studied by Arens in \cite{Arens52} and Michael in \cite%
{Michael52}. They were since studied by many authors under different names.
In particular, projective limits of projective families of C*-algebras were
studied by Inoue in \cite{Inoue71}, Apostol in \cite{Apostol71}, Schm\"{u}%
dgen in\textbf{\ }\cite{Schmuedgen75}, Phillips in \cite{Phillips88}, Bhatt
and Karia in \cite{BhattKaria93}, Fritsche in \cite{Fritzsche82}, etc. We
will follow Inoue \cite{Inoue71} in the usage of the name \textit{locally
C*-algebras} for these topological algebras. The current state of the basic
theory of locally C*-algebras is treated in the monograph of Fragoulopoulou 
\cite{Fragoulopoulou05}\textbf{.}

In particular, Inoue in \cite{Inoue71} proved that a complex locally
C*-algebra with a dense ideal posesses a bounded approximate identity which
belonges to the aforementioned ideal. On the other hand, Fritzsche showed in 
\cite{Fritzsche82} that if a unital complex locally C*-algebra has an
unbounded element then it also has a dense one-sided ideal.

Katz and Friedman in \cite{KatzFriedman06} introduced topological algebras
which are projective limits of projective families of real C*-algebras under
the name of \textit{real locally C*-algebras}, and projective limits of
projective families of JB-algebras under the name of \textit{locally
JB-algebras}.

The present paper is aimed to the presentation of analogues of the cited
above results of Inoue and Fritzsche for real locally C*-algebras and
locally JB-algebras.

\section{Preliminaries}

Let us first recall basic facts on JB algebras and fix the notation. (for
details see the monograph \cite{Hanche-OlsenStoermer84}). Let $A$ be
JB-algebra endowed with a product $".\circ ."$. We write 
\begin{equation*}
A_{\mathbf{1}}=\{a\in A;\left\Vert a\right\Vert \leq 1\},
\end{equation*}%
\begin{equation*}
A^{+}=\{a^{2};a\in A\},
\end{equation*}%
\begin{equation*}
A_{\mathbf{1}}^{+}=A_{\mathbf{1}}\cap A^{+}.
\end{equation*}%
For $a\in A$, the mapping 
\begin{equation*}
T_{a}:A\rightarrow A,
\end{equation*}%
is defined by putting 
\begin{equation*}
T_{a}b=a\circ b,
\end{equation*}%
and , 
\begin{equation*}
U_{a}:A\rightarrow A,
\end{equation*}%
is defined by putting 
\begin{equation*}
U_{a}b=2a\circ (a\circ b)-a^{2}\circ b.
\end{equation*}%
It is well known that 
\begin{equation*}
U_{a}(A^{+})\subset A^{+}.
\end{equation*}

A closed subspace $I$ of $A$ is called a \textit{Jordan ideal} if 
\begin{equation*}
T_{a}(A)\subset I,
\end{equation*}%
for all $a\in I$. Similarly, a closed subspace $Q$ of $A$ is said to be a 
\textit{quadratic ideal} if 
\begin{equation*}
U_{a}(A)\subset Q,
\end{equation*}%
for all $a\in Q$. Both the Jordan ideal and the quadratic ideal of a Jordan
algebra are subalgebras. We will use the symbol $B[a_{1},...,a_{n}]$ to
denote the JB-subalgebra $B$ of $A$ generated by elements $a_{1},...,a_{n}$.
The elements $a,b\in A$ are said to be \textit{operator commuting} if 
\begin{equation*}
T_{a}T_{b}=T_{b}T_{a}.
\end{equation*}

The JB-algebra $A$ is called \textit{associative or abelian} if it consists
of operator commuting elements. The associative subalgebra $B[a]$ is said to
be \textit{singly generated}.

The following two identities are corollaries from Shirshov-MacDonald theorem
(see \cite{Hanche-OlsenStoermer84} for details):

\begin{equation*}
(U_{x}y)^{2}=U_{x}U_{y}x^{2};
\end{equation*}

\begin{equation*}
U_{U_{x}y}z=U_{x}U_{y}U_{x}z.
\end{equation*}

The second identities is known by the name of \textit{MacDonald Identity}.

Now, let us briefly recall some more basic material from the aforementioned
sources one needs to comprehend what follows.

A Hausdorff topological vector space over the field of $%
\mathbb{R}
$ or $%
\mathbb{C}
$, in which any neighborhood of the zero element contains a convex
neighborhood of the zero element; in other words, a topological vector space
is a \textit{locally convex space} if and only if the topology of is a
Hausdorff locally convex topology.

A number of general properties of locally convex spaces follows immediately
from the corresponding properties of locally convex topologies; in
particular, subspaces and Hausdorff quotient spaces of a locally convex
space, and also products of families of locally convex spaces, are
themselves locally convex spaces. Let $\Lambda $ be an upward directed set
of indices and a family 
\begin{equation*}
\{E_{\alpha },\alpha \in \Lambda \},
\end{equation*}%
of locally convex spaces (over the same field) with topologies 
\begin{equation*}
\{\tau _{\alpha },\alpha \in \Lambda \}.
\end{equation*}%
Suppose that for any pair $(\alpha ,\beta )$, 
\begin{equation*}
\alpha \leq \beta ,
\end{equation*}%
$\alpha ,\beta \in \Lambda $, there is defined a continuous linear mapping 
\begin{equation*}
g_{\alpha }^{\beta }:E_{\beta }\rightarrow E_{\alpha }.
\end{equation*}

A family 
\begin{equation*}
\{E_{\alpha },\alpha \in \Lambda \}
\end{equation*}%
is called \textit{projective}, if for each triplet $(\alpha ,\beta ,\gamma
), $ 
\begin{equation*}
\alpha \leq \beta \leq \gamma ,
\end{equation*}%
$\alpha ,\beta ,\gamma \in \Lambda ,$ 
\begin{equation*}
g_{\alpha }^{\gamma }=g_{\beta }^{\gamma }\circ g_{\alpha }^{\beta },
\end{equation*}%
and for each $\alpha \in \Lambda ,$ 
\begin{equation*}
g_{\alpha }^{\alpha }=Id.
\end{equation*}

Let $E$\ be the subspace of the product 
\begin{equation*}
\dprod\limits_{\alpha \in \Lambda }E_{\alpha },
\end{equation*}%
whose elements 
\begin{equation*}
x=(x_{\alpha }),
\end{equation*}%
satisfy the relations 
\begin{equation*}
x_{\alpha }=g_{\alpha }^{\beta }(x_{\beta }),
\end{equation*}%
for all $\alpha \leq \beta $. The space $E$\ is called the \textit{%
projective limit} of the projective family $E_{\alpha },$ $\alpha \in
\Lambda ,$\ with respect to the family $(g_{\alpha }^{\beta }),$ $\alpha
,\beta \in \Lambda $\ and is denoted by 
\begin{equation*}
\lim g_{\alpha }^{\beta }E_{\beta },
\end{equation*}%
or 
\begin{equation*}
\underset{\longleftarrow }{\lim }E_{\alpha }.
\end{equation*}%
The topology of $E$\ is the \textit{projective topology} with respect to the
family 
\begin{equation*}
(E_{\alpha },g_{\alpha }^{\beta },\pi _{\alpha }),
\end{equation*}%
$\alpha \in \Lambda $, where $\pi _{\alpha },$ $\alpha \in \Lambda ,$\ is
the restriction to the subspace $E$\ of the projection 
\begin{equation*}
\widehat{\pi }_{\alpha }:\dprod\limits_{\beta \in \Lambda }E_{\beta
}\rightarrow E_{\alpha },
\end{equation*}%
and 
\begin{equation*}
\pi _{\beta }=g_{\alpha }^{\beta }\circ \pi _{\alpha },
\end{equation*}%
$\forall \alpha ,\beta \in \Lambda ,$

When you take instead of $E_{\alpha },$ $\alpha \in \Lambda ,$ a projective
family of algebras, *-algebras, Jordan algebras, etc., you naturally get a
correspondent algebra, *-algebra or Jordan algebra structure in the
projective limit algebra 
\begin{equation*}
E=\underset{\longleftarrow }{\lim }E_{\alpha }.
\end{equation*}

Let $E$ be a vector space. A real function $p:E\rightarrow 
\mathbb{R}
$ on $E$ is called a \textit{seminorm}, if:

\textit{1).} $p(x)\geq 0,$ $\forall x\in E;$

\textit{2).} $p(\lambda x)=\left\vert \lambda \right\vert p(x),$ $\forall
\lambda \in 
\mathbb{R}
$ or $%
\mathbb{C}
$, and $x\in E;$

\textit{3).} $p(x+y)\leq p(x)+p(y),$ $\forall x,y\in E.$

One can see that 
\begin{equation*}
p(\mathbf{0})=0.
\end{equation*}%
If 
\begin{equation*}
p(x)=0,
\end{equation*}%
implies 
\begin{equation*}
x=\mathbf{0},
\end{equation*}%
seminorm is called a \textit{norm and is usually denoted by }$\left\Vert
.\right\Vert $. If a space with a norm is complete, it is called a \textit{%
Banach space}.

Let $(E,p)$ be a seminormed space, and 
\begin{equation*}
N_{p}=\ker (p)=p^{-1}\{0\}.
\end{equation*}%
The quotient space $E/N_{p}$ is a linear space and the function 
\begin{equation*}
\left\Vert .\right\Vert _{p}:E/N_{p}\rightarrow 
\mathbb{R}
_{+}:
\end{equation*}%
\begin{equation*}
x_{p}=x+N_{p}\rightarrow \left\Vert x_{p}\right\Vert _{p}=p(x),
\end{equation*}%
is a well defined norm on $E/N_{p}$ induced by the seminorm $p.$ The
corresponding quotient normed space will be denoted by $E/N_{p},$ and the
Banach space completion of $E/N_{p}$ by $E_{p}.$ One can easily see that $%
E_{p}$ is the Hausdorff completion of the seminormed space $(E,p).$

The algebras considered below will be without the loss of generality unital.
If the algebra does not have an identity, it can be adjoint by the usual
unitialization procedure.

A Jordan algebra is an algebra $E$ in which the identities 
\begin{equation*}
x\circ y=y\circ x,
\end{equation*}%
\begin{equation*}
x^{2}\circ (y\circ x)=(x^{2}\circ y)\circ x,
\end{equation*}%
hold.

If $E$ is an algebra, the seminorm $p$ on $E$ compatible with the
multiplication of $E$, in the sense that 
\begin{equation*}
p(xy)\leq p(x)p(y),
\end{equation*}%
$\forall x,y\in E,$ is called \textit{submultiplicative }or \textit{%
m-seminorm}.

For submultiplicative seminorm on a Jordan algebra $E$, the following
inequality holds: 
\begin{equation*}
p(x\circ y)\leq p(x)p(y),
\end{equation*}%
$\forall x,y\in E.$ A seminorm on a Jordan algebra $E$\ is called \textit{%
fine}, if the following inequality holds: 
\begin{equation*}
p(x^{2}+y^{2})\geq p(x^{2}),
\end{equation*}%
$\forall x,y\in E.$

A \textit{Banach-Jordan algebra} is Jordan algebra which is as well a Banach
algebra.

Let $E$ be an algebra. A subset $U$ of $E$ is called \textit{multiplicative}
or \textit{idempotent}, if 
\begin{equation*}
UU\subseteq U,
\end{equation*}%
in the sense that $\forall x,y\in U,$ the product 
\begin{equation*}
xy\in U.
\end{equation*}

If $p$ is an m-seminorm on $E$ the unit semiball $U_{p}(1)$ corresponding to 
$p$, that is 
\begin{equation*}
U_{p}(1)=\{x\in E:p(x)\leq 1\},
\end{equation*}%
and one can see that this set is multiplicative. Moreover, $U_{p}(1)$ is an
absolutely-convex (balanced and convex).absorbing subset of $E.$ It is known
that given an absorbing absolutely-convex subset 
\begin{equation*}
U\subset E,
\end{equation*}%
the function%
\begin{equation*}
p_{U}:E\rightarrow 
\mathbb{R}
_{+}:
\end{equation*}%
\begin{equation*}
x\rightarrow p_{U}(x)=\inf \{\lambda >0:x\in \lambda U\},
\end{equation*}%
called \textit{gauge} or \textit{Minkowski functional} of $U,$ is a
seminorm. One can see that a real-valued function $p$ on the algebra $E$ is
an m-seminorm iff 
\begin{equation*}
p=p_{U},
\end{equation*}%
for some absorbing, absolutely-convex and multiplicative subset 
\begin{equation*}
U\subset E.
\end{equation*}%
In fact, one can take 
\begin{equation*}
U=U_{p}(1).
\end{equation*}

By \textit{topological algebra} we mean a topological vector space which is
also an algebra, such that the ring multiplication is separately continuous.
A topological algebra $E$ is often denoted by $(E,\tau ),$ where $\tau $ is
the topology of the underlying topological vector space of $E.$ The topology 
$\tau $ is determined by a \textit{fundamental }$0$\textit{-neighborhood
system}, say $\mathcal{B}$, consisting of absorbing, balanced sets with the
property 
\begin{equation*}
\forall \text{ }V\in \mathcal{B}\text{ }\exists \text{ }U\in \mathcal{B},
\end{equation*}%
satisfying the condition $U+U\subseteq V.$ Since translations by $y$ in $%
(E,\tau )$, i.e. the maps 
\begin{equation*}
x\rightarrow x+y:
\end{equation*}%
\begin{equation*}
(E,\tau )\rightarrow (E,\tau ),
\end{equation*}%
$y\in E,$ are homomorphisms, an $x$-neighborhood in $(E,\tau )$ is of the
form 
\begin{equation*}
x+V,
\end{equation*}%
with $V\in \mathcal{B}$. A closed, absorbing and absolutely-convex subset of
a topological algebra $(E,\tau )$ is called \textit{barrel}. An \textit{%
m-barrel} is a multiplicative barrel of $(E,\tau ).$

A \textit{locally convex algebra} is a topological algebra in which the
underlying topological vector space is a locally convex space. The topology $%
\tau $\ of a locally convex algebra $(E,\tau )$ is defined by a fundamental $%
0$-neighborhood system consisting of closed absolutely-convex sets.
Equivalently, the same topology $\tau $\ is determined by a family of
nonzero seminorms. Such a family, say 
\begin{equation*}
\Gamma =\{p_{\alpha }\},
\end{equation*}%
$\alpha \in \Lambda ,$ or, for distinction purposes 
\begin{equation*}
\Gamma _{E}=\{p_{\alpha }\},
\end{equation*}%
$\alpha \in \Lambda ,$ is always assumed without a loss of generality 
\textit{saturated}. That is, for any finite subset 
\begin{equation*}
F\subset \Gamma ,
\end{equation*}%
the seminorm 
\begin{equation*}
p_{F}(x)=\underset{p\in F}{\max }p(x),
\end{equation*}%
$x\in E,$ again belongs to $\Gamma .$ Saying that 
\begin{equation*}
\Gamma =\{p_{\alpha }\},
\end{equation*}%
$\alpha \in \Lambda ,$ is a \textit{defining family of seminorms} for a
locally convex algebra $(E,\tau ),$ we mean that $\Gamma $ is a saturated
family of seminorms defining the topology $\tau $ on $E.$ That is 
\begin{equation*}
\tau =\tau _{\Gamma },
\end{equation*}%
with $\tau _{\Gamma }$ completely determined by a fundamental $\mathbf{0}$%
-neighborhood system given by the $\varepsilon $-semiballs 
\begin{equation*}
U_{p}(\varepsilon )=\varepsilon U_{p}(\varepsilon )=\{x\in E:p(x)\leq
\varepsilon \},
\end{equation*}%
$\varepsilon >0,$ $p\in \Gamma $. More precisely, for each $\mathbf{0}$%
-neighborhood 
\begin{equation*}
V\subset (E,\tau ),
\end{equation*}%
there is an $\varepsilon $-semiball $U_{p}(\varepsilon )$, $\varepsilon >0,$ 
$p\in \Gamma ,$ such that 
\begin{equation*}
U_{p}(\varepsilon )\subseteq V.
\end{equation*}%
The neighborhoods $U_{p}(\varepsilon )$, $\varepsilon >0,$ $p\in \Gamma ,$
are called \textit{basic }$\mathbf{0}$\textit{-neighborhoods}.

A locally C*-algebra (real locally C*-algebra, resp. locally JB-algebra) is
a projective limit of projective family of C*-algebras (real C*-algebras,
resp. JB-algebras). This is equivalent for locally C*- and real locally
C*-algebras to the requirement that the family of defining continuous
seminorms be regular: 
\begin{equation*}
p(x^{\ast }x)=p(x)^{2},
\end{equation*}%
as well as for the real locally C*-algebra $R$: 
\begin{equation*}
R\cap iR=\{\mathbf{0}\}.
\end{equation*}%
In the case of locally JB-algebras this is equivalent to the requirement
that the family of defining continuous seminorms be fine and regular:

\begin{equation*}
p(x^{2}+y^{2})\geq p(x^{2}),
\end{equation*}
and

\begin{equation*}
p(x^{2})=p(x)^{2},
\end{equation*}%
$\forall p\in \Gamma ,$ $x,y\in E.$

For a locally C*-algebra (real locally C*-algebra, resp. locally JB-algebra) 
$E$, by the bounded part we mean the subalgebra 
\begin{equation*}
E_{b}=\{x\in E:\left\Vert x\right\Vert _{\infty }=\underset{p\in \Gamma (E)}{%
\sup }p(x)<\infty \}.
\end{equation*}

\section{BAI in real locally C*-algebras and locally JB-algebras with dense
ideals}

Let $(A,\tau )$ be a real or complex associative topological algebra and $%
(a_{\lambda }),$ $\lambda \in \Lambda ,$ be a net in $(A,\tau )$ such that 
\begin{equation*}
\underset{\lambda }{\lim }\text{ }xa_{\lambda }=x=\underset{\lambda }{\lim }%
\text{ }a_{\lambda }x,
\end{equation*}%
for any $x\in A.$

Such a net is called \textit{approximate identity} (abbreviated to ai) of $%
(A,\tau ).$ If only the left or the right equality above is valid, then we
speak of a \textit{left} (resp. \textit{right}) \textit{ai}.\ In the case an 
\textit{ai}, when $(a_{\lambda }),$ $\lambda \in \Lambda ,$ of $(A,\tau )$
is a bounded subset of $(A,\tau ),$ we speak about a \textit{bounded
approximate identity} (abbreviated to \textit{bai}) of $(A,\tau ).$ In the
case an \textit{left} \textit{ai }(resp.\textit{\ right ai}), when $%
(a_{\lambda }),$ $\lambda \in \Lambda ,$ of $(A,\tau )$ is a bounded subset
of $(A,\tau ),$ we speak about a \textit{bounded left approximate identity} (%
\textit{bounded right approximate identity}), abbreviated to \textit{blai }%
(resp. \textit{brai}) of $(A,\tau ).$

Let now $(J,\tau )$ be a real Jordan topological algebra and $(a_{\lambda
}), $ $\lambda \in \Lambda ,$ be a net in $(J,\tau )$ such that 
\begin{equation*}
\underset{\lambda }{\lim }\text{ }a_{\lambda }\circ x=x,
\end{equation*}%
for any $x\in J.$

Such a net is called \textit{approximate identity} (abbreviated to ai) of $%
(J,\tau ).$In the case an \textit{ai}, when\textit{\ }$(a_{\lambda }),$ $%
\lambda \in \Lambda ,$ of $(J,\tau )$ is a bounded subset of $(J,\tau ),$ we
speak about a \textit{bounded approximate identity} (abbreviated to \textit{%
bai}) of $(J,\tau ).$

Let now $(a_{\lambda }),$ $\lambda \in \Lambda ,$ be a net in $(J,\tau )$
such that%
\begin{equation*}
\underset{\lambda }{\lim }\text{ }U_{a_{\lambda }}x=x,
\end{equation*}%
for any $x\in J,$ then we speak of a \textit{quadratic ai }(abbreviated to 
\textit{qai}). In the case a \textit{qai}, $(a_{\lambda }),$ $\lambda \in
\Lambda ,$ of $(J,\tau )$ is a bounded subset of $(J,\tau ),$ we speak about
a \textit{bounded quadratic approximate identity }(abbreviated to \textit{%
bqai) }of $(J,\tau )$.

Taking a completion of an associative real or complex topological algebra $%
(A,\tau ),$ or a completion of a real Jordan topological algebra $(J,\tau )$
(that is taking the completion of the underlying topological vector space of 
$(A,\tau )$ (resp. $(J,\tau )$)), one may fail to get a topological algebra,
unless the multiplication in $(A,\tau )$ (resp. $(J,\tau )$) is jointly
continuous (see, for example for \cite{Mallios86}\ details). If 
\begin{equation*}
\tau =\tau _{\Gamma },
\end{equation*}%
the respective completion of $(A,\tau )$ (resp. $(J,\tau )$), when it
exists, will be denoted by $(\widetilde{A},\tau _{\widetilde{\Gamma }})$
(resp. $(\widetilde{J},\tau _{\widetilde{\Gamma }})$), where $\widetilde{%
\Gamma }$ consists of the (unique) extentions of the elements of $\Gamma $
to the corresponding completion of $(A,\tau )$ (resp. $(J,\tau )$).

The following three lemmas present some properties of approximate
identities. The first one talks about \textit{bai}'s for completion algebras
and reveals the fact that the squares of the elements of \textit{bai} make
up a \textit{bai} as well.

\begin{lemma}
Let $(a_{\lambda }),$ $\lambda \in \Lambda ,$ be a bai of a real or complex
associative topological algebra $(A,\tau ),$ or Jordan topological algebra $%
(J,\tau ),$ with continuous multiplication. Then:

i). $(a_{\lambda }),$ $\lambda \in \Lambda ,$ is also bai for the completion 
$(\widetilde{A},\widetilde{\tau })$\ of $(A,\tau )$ (resp. $(\widetilde{J},%
\widetilde{\tau })$ of $(J,\tau )$);

ii). $(a_{\lambda }^{2}),$ $\lambda \in \Lambda ,$ is a bai of $(A,\tau )$
(resp. $(J,\tau )$).
\end{lemma}

\begin{proof}
Immediately follows from considerations in \cite{Mallios86}.
\end{proof}

The second one talks about \textit{bai} made up out of adjoint elements of
another \textit{bai} in a topological *-algebra.

\begin{lemma}
Let $(A,\tau )$ be a topological *-algebra with an ai $(a_{\lambda }),$ $%
\lambda \in \Lambda .$ Then $(a_{\lambda }^{\ast }),$ $\lambda \in \Lambda ,$
is an ai of $(A,\tau ).$ Moreover, $(a_{\lambda }^{\ast }),$ $\lambda \in
\Lambda ,$ is a bai whenever $(a_{\lambda }),$ $\lambda \in \Lambda ,$ is a
bai.
\end{lemma}

\begin{proof}
Immediately follows from considerations in \cite{Mallios86}.
\end{proof}

The third one talks about the properties of \textit{bai} in certain
factor-algebras of some topological *-algebras and topological Jordan
algebras. Let $(A,\tau )$ be a real or complex topological *-algebra (or $%
(J,\tau )$ be a real Jordan topological algebra) with a saturated separating
family of seminorms. It is called an \textit{m-convex algebra}, if each
seminorm satisfies the \textit{submultiplicativity }inequality 
\begin{equation*}
p(xy)\leq p(x)p(y),
\end{equation*}%
for every $x,y\in A$ (resp. 
\begin{equation*}
p(x\circ y)\leq p(x)p(y),
\end{equation*}%
for every $x,y\in J$). It is called an \textit{m*-convex algebra}, if it is
m-convex and each seminorm satisfies the identity 
\begin{equation*}
p(x^{\ast })=p(x),
\end{equation*}%
for every $x\in A.$

\begin{lemma}
Let $(A,\tau )$ be a m*-convex algebra (or a Jordan topological m-convex
algebra $(J,\tau )$). Then, if $(a_{\lambda }),$ $\lambda \in \Lambda $, is
an ai, the net $(a_{\lambda ,p}),$ $\lambda \in \Lambda ,$ with 
\begin{equation*}
a_{\lambda ,p}\equiv a_{\lambda }+N_{p},
\end{equation*}%
$(a_{\lambda ,p})\in A$ (resp. $(a_{\lambda ,p})\in J$)$,$ $p\in \Gamma ,$ $%
\lambda \in \Lambda ,$ is an ai for both $(A,p)/N_{p}$ and $A_{p}$ (resp.
for both $(J,p)/N_{p}$ and $J_{p}$)$,$ for every $p\in \Gamma .$ Moreover, $%
(a_{\lambda ,p}),$ $\lambda \in \Lambda ,$ is bounded, whenever $(a_{\lambda
}),$ $\lambda \in \Lambda ,$ is bounded.
\end{lemma}

\begin{proof}
Immediately follows from considerations in \cite{Mallios86}.
\end{proof}

The following result is a real version of the theorem of Inoue (see \cite%
{Inoue71} for details).

\begin{theorem}
Let $(A,\tau )$ be a real locally C*-algebra, and $I$ be a dense ideal in $%
(A,\tau ).$ Then, $(A,\tau )$ has an ai $(a_{\lambda }),$ $\lambda \in
\Lambda ,$ consisting of elements of $I,$ such that:

1). The net $(a_{\lambda }),$ $\lambda \in \Lambda ,$ is increasing, in the
sense that 
\begin{equation*}
a_{\lambda }\geq \mathbf{0},
\end{equation*}%
for every $\lambda \in \Lambda ,$ and 
\begin{equation*}
a_{\lambda }\leq a_{\nu },
\end{equation*}%
for any $\lambda \leq \nu $ in $\Lambda .$

2). 
\begin{equation*}
p(a_{\lambda })\leq 1,
\end{equation*}%
for all $p\in \Gamma ,$ $\lambda \in \Lambda .$
\end{theorem}

\begin{proof}
Let $(A,\tau )$ be a real locally C*-algebra. Then, from \cite%
{KatzFriedman06} it follows, that $(A_{%
\mathbb{C}
},\widehat{\tau })$ is a complex locally C*-algebra, where 
\begin{equation*}
A_{%
\mathbb{C}
}=A\dotplus iA,
\end{equation*}%
is the complexification of $A.$ One can easily see that if $I$ be a dense
ideal in $(A,\tau ),$ then $I_{%
\mathbb{C}
}$ is a dense ideal of $(A_{%
\mathbb{C}
},\widehat{\tau }),$ where 
\begin{equation*}
I_{%
\mathbb{C}
}=I\dotplus iI,
\end{equation*}%
the complexification of $I$ in $(A_{%
\mathbb{C}
},\widehat{\tau }).$ According to \cite{Inoue71}, there exists an ai $%
(u_{\lambda }),$ $\lambda \in \Lambda ,$ in $(A_{%
\mathbb{C}
},\widehat{\tau })$\ consisting of elements of $I_{%
\mathbb{C}
},$ such that: 

1). The net $(u_{\lambda }),$ $\lambda \in \Lambda ,$ is increasing, in the
sense that 
\begin{equation*}
u_{\lambda }\geq \mathbf{0},
\end{equation*}%
for every $\lambda \in \Lambda ,$ and 
\begin{equation*}
u_{\lambda }\leq u_{\nu },
\end{equation*}%
for any $\lambda \leq \nu $ in $\Lambda ;$

2). 
\begin{equation*}
\widehat{p}(u_{\lambda })\leq 1,
\end{equation*}%
for all $\widehat{p}\in \widehat{\Gamma },$ $\lambda \in \Lambda ,$ where
for each $p\in \Gamma ,$ defined on $A,$ its extention $\widehat{p}\in 
\widehat{\Gamma }$ on all $A_{%
\mathbb{C}
},$ is defined as: 
\begin{equation*}
\widehat{p}(x+iy)=\sqrt{p(x)^{2}+p(y)^{2}},
\end{equation*}%
for every $x,y\in A.$

Let now 
\begin{equation*}
u_{\lambda }=a_{\lambda }+ib_{\lambda },
\end{equation*}%
$\lambda \in \Lambda ,$ where $a_{\lambda },b_{\lambda }\in A.$ One can see
that the net $(a_{\lambda }),$ $\lambda \in \Lambda ,$ satisfies the
required conditions.
\end{proof}

The following result is a real version of Inoue's theorem from \cite{Inoue71}%
\ on existence of \textit{left} (resp. \textit{right}) \textit{ai} in real
locally C*-algebras with dense left (resp. right) ideals.

\begin{theorem}
Let $(A,\tau )$ be a real locally C*-algebra, and $I$ be a dense left (resp.
right) ideal in $(A,\tau ).$ Then, $(A,\tau )$ has a left (resp. right) ai $%
(a_{\lambda }),$ $\lambda \in \Lambda ,$ consisting of elements of $I,$ such
that:

1). The net $(a_{\lambda }),$ $\lambda \in \Lambda ,$ is increasing, in the
sense that 
\begin{equation*}
a_{\lambda }\geq \mathbf{0},
\end{equation*}%
for every $\lambda \in \Lambda ,$ and 
\begin{equation*}
a_{\lambda }\leq a_{\nu },
\end{equation*}%
for any $\lambda \leq \nu $ in $\Lambda .$

2). 
\begin{equation*}
p(a_{\lambda })\leq 1,
\end{equation*}%
for all $p\in \Gamma ,$ $\lambda \in \Lambda .$
\end{theorem}

\begin{proof}
One can note that a complexification of a dense left (resp. right) ideal in
in $(A,\tau )$ is a dense left (resp. right) ideal in $(A_{%
\mathbb{C}
},\widehat{\tau }).$ With that it mind, the rest of the proof repeits the
proof of the preseeding theorem.
\end{proof}

Now we turn our attention to the case of Jordan algebras. The next result is
a version of the theorem of Inoue from \cite{Inoue71} on existence of 
\textit{bai} for locally JB-algebras with dense Jordan idals.

\begin{theorem}
Let $(J,\tau )$ be a locally JB-algebra, and $I$ be a dense ideal in $%
(J,\tau ).$ Then, $(J,\tau )$ has an ai $(a_{\lambda }),$ $\lambda \in
\Lambda ,$ consisting of elements of $I,$ such that:

1). The net $(a_{\lambda }),$ $\lambda \in \Lambda ,$ is increasing, in the
sense that 
\begin{equation*}
a_{\lambda }\geq \mathbf{0},
\end{equation*}%
for every $\lambda \in \Lambda ,$ and 
\begin{equation*}
a_{\lambda }\leq a_{\nu },
\end{equation*}%
for any $\lambda \leq \nu $ in $\Lambda .$

2). 
\begin{equation*}
p(a_{\lambda })\leq 1,
\end{equation*}%
for all $p\in \Gamma ,$ $\lambda \in \Lambda .$
\end{theorem}

\begin{proof}
Let us first consider the set 
\begin{equation*}
\Lambda =\{F\subseteq I:F\text{- \textit{finite}}\},
\end{equation*}%
ordered by inclusion. For each 
\begin{equation*}
\lambda =\{x_{1},x_{2},...,x_{n}\},
\end{equation*}%
we put 
\begin{equation*}
b_{\lambda }=\dsum\limits_{i=1}^{n}x_{i}^{2}.
\end{equation*}%
For what follows we need a definition and a few lemmas about positive
elements and spectrum in $J$. 

If $J$ is a unital locally JB-algebra, and $x\in J,$ we denote by $C(x)$ the
smallest locally JB-subalgebra containing $x$ and $\mathbf{1}$. According to
Shirshov-Cohn theorem (\cite{Hanche-OlsenStoermer84}), $C(x)$ is
associative. We define the \textit{spectrum} of $x$ in $J,$ denoted by $%
sp_{J}(x)$, 
\begin{equation*}
sp_{J}(x)=\{\alpha \in 
\mathbb{R}
:(x-\alpha \mathbf{1)}\text{ \textit{doesn't have an inverse in} }C(x)\}.
\end{equation*}

When the algebra is not unital, we first adjoint a unit (\cite%
{Hanche-OlsenStoermer84}), and then compute the spectrum in the unitization.

An element $x\in J$ is called \textit{positive}, and we write 
\begin{equation*}
x\geq \mathbf{0},
\end{equation*}%
if 
\begin{equation*}
sp_{J}(x)\subseteq \lbrack 0,\infty ).
\end{equation*}%
We denote by $J_{+}$ the set of all positive elements in $J$.

\begin{lemma}
Let $(J,\tau )$ be a locally JB-algebra, and 
\begin{equation*}
J=\underset{\longleftarrow }{\lim }J_{p},
\end{equation*}%
$p\in \Gamma ,$ be the Arens-Michael decomposition of $J$ as a projective
limit of a projective family of JB-algebras, and 
\begin{equation*}
\pi _{p}:J\longrightarrow J_{p},
\end{equation*}%
be the continuous projection of $J$ onto $J_{p},$ for each $p\in \Gamma .$
The following conditions are equivalent for $x\in J$: 

1). 
\begin{equation*}
x\geq \mathbf{0};
\end{equation*}

2). 
\begin{equation*}
x=y^{2},
\end{equation*}%
for some $y\in J;$

3). 
\begin{equation*}
x_{p}\geq \mathbf{0}_{p},
\end{equation*}%
for each 
\begin{equation*}
x_{p}=\pi _{p}(x)\in J_{p},
\end{equation*}%
$p\in \Gamma $, and $\mathbf{0}_{p}$ is the zero-element of $J_{p}.$
\end{lemma}

\begin{proof}
Easily follows from Arens-Michael decompostion and correspondent properties
of JB-algebras (see \cite{Hanche-OlsenStoermer84}, \cite{KatzFriedman06}).
\end{proof}

\begin{corollary}
Let $(J,\tau )$ be a locally JB-algebra. Then 
\begin{equation*}
J_{+}=\{x^{2}:x\in J\}.
\end{equation*}
\end{corollary}

\begin{proof}
Evident.
\end{proof}

\begin{corollary}
Let $(J,\tau )$ be a locally JB-algebra. Then $J$ is formally real.
\end{corollary}

\begin{proof}
Easily follows from Arens-Michael decomposition and the correspondent
property of JB-algebras.
\end{proof}

\begin{lemma}
Let $(J,\tau )$ be a locally JB-algebra. Then $J_{+}$ is a closed convex
cone, such that 
\begin{equation*}
J_{+}\cap (-J_{+})=\{\mathbf{0}\}.
\end{equation*}
\end{lemma}

\begin{proof}
Clearly follows from Arens-Michael decomposition and correspondent fact for
JB-algebras (\cite{Hanche-OlsenStoermer84}, \cite{KatzFriedman06}).
\end{proof}

\begin{corollary}
Let $(J,\tau )$ be a locally JB-algebra, and $x,y\in J.$ The following
statements hold:

1). 
\begin{equation*}
x\leq y\Longrightarrow U_{z}x\leq U_{z}y,
\end{equation*}
for all $z\in J;$

2). 
\begin{equation*}
\mathbf{0}\leq x\leq y\Longrightarrow p(x)\leq p(y),
\end{equation*}%
for all $p\in \Gamma ;$

3). 
\begin{equation*}
\mathbf{0}\leq x\leq y\Longrightarrow \mathbf{0}\leq x^{1/2}\leq y^{1/2}.
\end{equation*}

In the case when $(J,\tau )$ is unital, one has:

4). 
\begin{equation*}
x>\mathbf{0}\Longrightarrow x\in G_{J},
\end{equation*}%
where $G_{J}$ is the set of invertible elements in $J$;

5). 
\begin{equation*}
x\geq \mathbf{1}\Longrightarrow x^{-1}\leq \mathbf{1};
\end{equation*}

6). 
\begin{equation*}
\mathbf{0}<x\leq y\Longrightarrow y^{-1}\leq x^{-1}.
\end{equation*}
\end{corollary}

\begin{proof}
All properties follow from Arens-Michael decomposition of the algebra $%
(J,\tau )$\ and corresponding properties of JB-algebras (\cite%
{Hanche-OlsenStoermer84}, \cite{KatzFriedman06}).
\end{proof}

Now, based on the presiding lemmas and corollaries, one can easily see that 
\begin{equation*}
b_{\lambda }\in I\cap J_{+},
\end{equation*}%
for every $\lambda \in \Lambda $. 

Now, we need the following lemma.

\begin{lemma}
Let $(J,\tau )$ be a non-unital locally JB-algebra, and 
\begin{equation*}
J=\underset{\longleftarrow }{\lim }J_{p},
\end{equation*}%
$p\in \Gamma ,$ be the Arens-Michael decomposition of $J$ as a projective
limit of a projective family of JB-algebras, and 
\begin{equation*}
\pi _{p}:J\longrightarrow J_{p},
\end{equation*}%
be the continuous projection of $J$ onto $J_{p},$ for each $p\in \Gamma .$
There exits a unique unital locally JB-algebra $(J_{\mathbf{1}},\tau
_{1}^{\prime }),$ such that $(J,\tau )$ is a locally JB-subalgebra, and 
\begin{equation*}
J_{\mathbf{1}}=\underset{\longleftarrow }{\lim }J_{\mathbf{1},p^{\prime }},
\end{equation*}%
$p^{\prime }\in \Gamma ^{\prime },$ is the Arens-Michael decomposition of $%
J_{\mathbf{1}}$ as a projective limit of a projective family of unital
JB-algebras $J_{\mathbf{1},p}$, and each $J_{\mathbf{1},p}$ is the
unitization of a correspondent $J_{p},$ for each $p^{\prime }\in \Gamma
^{\prime }$ is the extension of a correspondent $p\in \Gamma .$
\end{lemma}

\begin{proof}
Easily obtained using a combination of arguments in \cite{KatzFriedman06}
and \cite{Hanche-OlsenStoermer84}.
\end{proof}

Let now $M$ be the locally JB-subalgebra of $(J_{\mathbf{1}},\tau
_{1}^{\prime }),$ generated by two elements- $b_{\lambda },$ and $\mathbf{1}$%
. Accorging to Shirshov-Cohn theorem (see \cite{Hanche-OlsenStoermer84}),
this subalgebra is associative. If $\ $%
\begin{equation*}
S\equiv sp_{J}(b_{\lambda })=sp_{J_{\mathbf{1}}}(b_{\lambda })\subseteq
\lbrack 0,\infty ),
\end{equation*}%
and 
\begin{equation*}
f(t)=t(t+\frac{1}{n})^{-1},
\end{equation*}%
for every $t\in 
\mathbb{R}
,$ $n\in 
\mathbb{N}
,$ we obtain that 
\begin{equation*}
f|_{S}\in C(S).
\end{equation*}

Accorging to Spectral theorem for locally JB-algebras (see \cite%
{KatzFriedman06}), $C(S)$ is embedded in $M$ by means of a unique
topological injective morphism $\Phi ,$ such that 
\begin{equation*}
\Phi (\mathbf{1}_{C(S)})=\mathbf{1,}
\end{equation*}%
and 
\begin{equation*}
\Phi (id_{S})=x,
\end{equation*}%
where $\mathbf{1}_{C(S)}$ is the constant function $1$ on $S$, and $id_{S}$
is the identity map of $S.$ Therefore, we can define 
\begin{equation*}
a_{\lambda }=\Phi (f|_{S})=b_{\lambda }\circ (b_{\lambda }+\frac{\mathbf{1}}{%
n})^{-1}\in M,
\end{equation*}%
$\lambda \in \Lambda .$

Now we need the following lemma.

\begin{lemma}
Let $(P,\tau _{P})$ and $(Q,\tau _{Q})$ be two locally JB-algebras, and 
\begin{equation*}
\varphi :(P,\tau _{P})\longrightarrow (Q,\tau _{Q}),
\end{equation*}%
be a Jordan morphism. Then, 
\begin{equation*}
\varphi (P_{+})=Q_{+}\cap \varphi (P).
\end{equation*}
\end{lemma}

\begin{proof}
Obvious.
\end{proof}

\begin{corollary}
Let $(P,\tau _{P})$ be a locally JB-algebra, and $Q$ be a closed Jordan
subalgebra of $(P,\tau _{P})$, with 
\begin{equation*}
\tau _{Q}=\tau _{P}|_{Q}.
\end{equation*}%
Then 
\begin{equation*}
Q_{+}=P_{+}\cap Q.
\end{equation*}
\end{corollary}

\begin{proof}
Obvious.
\end{proof}

The presiding corollary implies that 
\begin{equation*}
b_{\lambda }+\frac{\mathbf{1}}{n}>\mathbf{0,}
\end{equation*}%
in $M.$ From Corollary 3.4 it follows that 
\begin{equation*}
(b_{\lambda }+\frac{\mathbf{1}}{n})\in M,
\end{equation*}%
is invertible in $M$. Since 
\begin{equation*}
0\leq f|_{S}\leq \mathbf{1}_{C(S)},
\end{equation*}%
the presiding lemma implies that 
\begin{equation*}
\mathbf{0}\leq a_{\lambda }\leq \mathbf{1},
\end{equation*}%
$\lambda \in \Lambda $.

One now can see that 
\begin{equation*}
a_{\lambda }\in I\cap J_{+},
\end{equation*}%
for all $\lambda \in \Lambda ,$ and 
\begin{equation*}
p^{\prime }(a_{\lambda })=p(a_{\lambda })\leq 1,
\end{equation*}%
for all $p\in \Gamma ,$ $p^{\prime }\in \Gamma ^{\prime },$ and $\lambda \in
\Lambda .$

A computation shows that 
\begin{equation*}
\dsum\limits_{i=1}^{n}((a_{\lambda }-\mathbf{1})\circ
x_{i})^{2}=U_{(a_{\lambda }-\mathbf{1)}}b_{\lambda }=\mathbf{(}a_{\lambda }-%
\mathbf{1)}\text{ }\mathbf{\circ }\text{ }b_{\lambda }\circ (a_{\lambda }-%
\mathbf{1)=}\text{ }n^{-2}b_{\lambda }\circ (b_{\lambda }+\frac{\mathbf{1}}{n%
})^{-2}.
\end{equation*}

Now, taking a function 
\begin{equation*}
g(t)=t(t+\frac{1}{n})^{-2},
\end{equation*}%
for every $t\in 
\mathbb{R}
,$ $n\in 
\mathbb{N}
,$ one can see that it has a maximum value at 
\begin{equation*}
t=\frac{1}{n},
\end{equation*}%
so that 
\begin{equation*}
0\leq g|_{S}\leq \frac{n}{4}\mathbf{1}_{C(S)}.
\end{equation*}%
Therefore, we get 
\begin{equation*}
0\leq \Phi (g|_{S})=b_{\lambda }\circ (b_{\lambda }+\frac{\mathbf{1}}{n}%
)^{-2}\leq \frac{n}{4}\mathbf{1},
\end{equation*}%
using presiding calculations, we obtain 
\begin{equation*}
((a_{\lambda }-\mathbf{1})\circ x_{i})^{2}\leq \frac{1}{4n}\mathbf{1,}
\end{equation*}%
for every $i=1,...,n.$ Applying Corollary 3.2 we get 
\begin{equation*}
p^{\prime }((a_{\lambda }-\mathbf{1})\circ x_{i})^{2}=p(a_{\lambda }\circ
x_{i}-x_{i})^{2}\leq \frac{1}{4n},
\end{equation*}%
for all $p\in \Gamma ,$ $p^{\prime }\in \Gamma ^{\prime },$ $i=1,...,n.$

Let now $\varepsilon $ be an arbitrary small positive real number, and $x\in
I.$ Let $\lambda (\varepsilon )$ be a finite subset of I with n elements,
such that $x\in \lambda (\varepsilon ),$ and 
\begin{equation*}
n>\frac{1}{\varepsilon ^{2}}.
\end{equation*}%
Then, based on presiding inequalities, we get that 
\begin{equation*}
p(a_{\lambda }\circ x-x)<\varepsilon ,
\end{equation*}%
for every 
\begin{equation*}
\lambda \geq \lambda (\varepsilon ),
\end{equation*}%
and $p\in \Gamma .$ Thus we obtain that 
\begin{equation*}
\underset{\lambda }{\lim }a_{\lambda }\circ x=x,
\end{equation*}%
for every $x\in I,$ and, because $I$ is dense in $(J,\tau ),$ and 
\begin{equation*}
p(a_{\lambda })\leq 1,
\end{equation*}%
for any $\lambda \in \Lambda ,$ and $p\in \Gamma ,$ we get that  
\begin{equation*}
\underset{\lambda }{\lim }a_{\lambda }\circ x=x,
\end{equation*}%
for every $x\in J.$

It remains to show that 
\begin{equation*}
a_{\lambda }\leq a_{\nu },
\end{equation*}%
for any 
\begin{equation*}
\lambda \leq \nu ,
\end{equation*}%
$\lambda ,\nu \in \Lambda .$ Let 
\begin{equation*}
\lambda =\{x_{1},...,x_{n}\},
\end{equation*}%
and 
\begin{equation*}
\nu =\{x_{1},...,x_{m}\},
\end{equation*}%
be in $\Lambda $ with $n\leq m.$ Then 
\begin{equation*}
b_{\nu }-b_{\lambda }=\dsum\limits_{i=n+1}^{m}x_{i}^{2}\in J_{+}.
\end{equation*}%
Moreover, 
\begin{equation*}
\mathbf{0}<b_{\lambda }+\frac{\mathbf{1}}{n}\leq b_{\nu }+\frac{\mathbf{1}}{n%
},
\end{equation*}%
therefore, based on presiding considerations we obtain that 
\begin{equation*}
(b_{\nu }+\frac{\mathbf{1}}{n})^{-1}\leq (b_{\lambda }+\frac{\mathbf{1}}{n}%
)^{-1}.
\end{equation*}%
One can notice now that for real non-negative $t,$ since $n\leq m,$ 
\begin{equation*}
\frac{1}{n}(t+\frac{1}{n})^{-1}\geq \frac{1}{m}(t+\frac{1}{m})^{-1}.
\end{equation*}%
Therefore, from the Spectral theorem it follows that 
\begin{equation*}
\frac{1}{n}(b_{\nu }+\frac{\mathbf{1}}{n})^{-1}\geq \frac{1}{m}(b_{\nu }+%
\frac{\mathbf{1}}{m})^{-1},
\end{equation*}%
and finally 
\begin{equation*}
a_{\lambda }=\mathbf{1}-\frac{1}{n}(b_{\lambda }+\frac{\mathbf{1}}{n}%
)^{-1}\leq \mathbf{1}-\frac{1}{n}(b_{\nu }+\frac{\mathbf{1}}{n})^{-1}\leq 
\mathbf{1}-\frac{1}{m}(b_{\nu }+\frac{\mathbf{1}}{m})^{-1}=a_{\nu }.
\end{equation*}
\end{proof}

The following result is a version of Inoue's theorem for existence of 
\textit{bqai} in locally JB-algebras with dense quadratic ideals.

\begin{theorem}
Let $(J,\tau )$ be a locally JB-algebra, and $I$ be a dense quadratic ideal
in $(J,\tau ).$ Then, $(J,\tau )$ has an qai $(a_{\lambda }),$ $\lambda \in
\Lambda ,$ consisting of elements of $I,$ such that:

1). The net $(a_{\lambda }),$ $\lambda \in \Lambda ,$ is increasing, in the
sense that 
\begin{equation*}
a_{\lambda }\geq \mathbf{0},
\end{equation*}%
for every $\lambda \in \Lambda ,$ and 
\begin{equation*}
a_{\lambda }\leq a_{\nu },
\end{equation*}%
for any $\lambda \leq \nu $ in $\Lambda .$

2). 
\begin{equation*}
p(a_{\lambda })\leq 1,
\end{equation*}%
for all $p\in \Gamma ,$ $\lambda \in \Lambda .$
\end{theorem}

\begin{proof}
Follows step-by-step the proof of the presiding Theorem 3.
\end{proof}

\section{Existence of dense ideals in real locally C*-algebras and locally
JB-algebras with unbounded elements}

In \cite{Fritzsche82} Fritzsche established that that if a complex unital
locally C*-algebra has an unbounded element then it also has a proper dense
left (resp. right) ideal.

The next result is a real analogue of the theorem of Fritzsche from \cite%
{Fritzsche82} for real locally C*-algebras.

\begin{theorem}
Let $(A,\tau )$ be a real unital locally C*-algebra, and $x\in A$, be such
that $x\notin A_{b},$ where 
\begin{equation*}
A_{b}=\{x\in A:\left\Vert x\right\Vert _{\infty }=\underset{p\in \Gamma (E)}{%
\sup }p(x)<\infty \}.
\end{equation*}%
Then $(A,\tau )$ has a proper dense left (resp. right) ideal $I$.
\end{theorem}

\begin{proof}
Let $(A,\tau )$ be a real locally C*-algebra. Then, from \cite%
{KatzFriedman06} it follows, that $(A_{%
\mathbb{C}
},\widehat{\tau })$ is a complex locally C*-algebra, where 
\begin{equation*}
A_{%
\mathbb{C}
}=A\dotplus iA,
\end{equation*}%
is the complexification of $A.$ If $x\in A$, be such that $x\notin A_{b},$
then there exists at least one unbounded element in $(A_{%
\mathbb{C}
},\widehat{\tau }),$ for example, one can take 
\begin{equation*}
(\mathbf{1}+i)x=x+ix.
\end{equation*}%
Therefore, from Fritzsche theorem for complex locally C*-algebras (see \cite%
{Fritzsche82}) it follows that there exists a dense left (resp. right) ideal 
$I_{%
\mathbb{C}
}$ in $(A_{%
\mathbb{C}
},\widehat{\tau }).$ Then, from the theorem of Inoue (\cite{Inoue71}) it
follows that there exists a left (resp. right) approximative identity $%
(u_{\lambda }),$ $\lambda \in \Lambda ,$ consisting of elements of $I_{%
\mathbb{C}
},$ such that: 

1). The net $(u_{\lambda }),$ $\lambda \in \Lambda ,$ is increasing, in the
sense that 
\begin{equation*}
u_{\lambda }\geq \mathbf{0},
\end{equation*}%
for every $\lambda \in \Lambda ,$ and 
\begin{equation*}
u_{\lambda }\leq u_{\nu },
\end{equation*}%
for any $\lambda \leq \nu $ in $\Lambda .$

2). 
\begin{equation*}
\widehat{p}(u_{\lambda })\leq 1,
\end{equation*}%
for all $\widehat{p}\in \widehat{\Gamma },$ $\lambda \in \Lambda .$where for
each $p\in \Gamma ,$ defined on $A,$ its extention $\widehat{p}\in \widehat{%
\Gamma }$ on all $A_{%
\mathbb{C}
},$ is defined as: 
\begin{equation*}
\widehat{p}(x+iy)=\sqrt{p(x)^{2}+p(y)^{2}},
\end{equation*}%
for every $x,y\in A.$

Let now 
\begin{equation*}
u_{\lambda }=a_{\lambda }+ib_{\lambda },
\end{equation*}%
$\lambda \in \Lambda ,$ where $a_{\lambda },b_{\lambda }\in A.$ Let us now
show that $(a_{\lambda }),$ $\lambda \in \Lambda ,$ is a left (resp. right)
approximate identity in $A.$ one can easily see that $(a_{\lambda }),$ $%
\lambda \in \Lambda ,$ is an increasing net in $A_{+},$ and, in fact, on one
hand, we have 
\begin{equation*}
p(a_{\lambda })\leq \sqrt{p(a_{\lambda })^{2}+p(b_{\lambda })^{2}}=\widehat{p%
}(a_{\lambda }+ib_{\lambda })=\widehat{p}(u_{\lambda })\leq 1,
\end{equation*}%
On the other hand, we have 
\begin{equation*}
p(xa_{\lambda }-x)\leq \widehat{p}(xu_{\lambda }-x)\longrightarrow 0,
\end{equation*}%
(resp. 
\begin{equation*}
p(a_{\lambda }x-x)\leq \widehat{p}(u_{\lambda }x-x)\longrightarrow 0\text{)},
\end{equation*}%
which proves that $(a_{\lambda }),$ $\lambda \in \Lambda $, is a left (resp.
right) approximate identity in $A$.

Let now 
\begin{equation*}
I=\{x:x\in A,\text{ }x=yz,\text{ where }y\in \{(a_{\lambda })\setminus 
\mathbf{1}\},\text{ }\lambda \in \Lambda ,\text{ and }z\in A\}
\end{equation*}%
(resp. 
\begin{equation*}
I=\{x:x\in A,\text{ }x=zy,\text{ where }y\in \{(a_{\lambda })\setminus 
\mathbf{1}\},\text{ }\lambda \in \Lambda ,\text{ and }z\in A\}\text{)}.
\end{equation*}

Due to associativity of multiplication in $A,$ I is obviously a left (resp.
right) ideal. On the other hand, one can see that $I$ is dense in $A$ due to
the fact that $(a_{\lambda }),$ $\lambda \in \Lambda $, is a left (resp.
right) approximate identity in $A.$ In addition, one can see that $I$ is
proper, because for the unbounded element $x\in A,$ 
\begin{equation*}
p(yx-x)\longrightarrow 0,
\end{equation*}%
when 
\begin{equation*}
y\in \{(a_{\lambda })\setminus \mathbf{1}\},\text{ }\lambda \in \Lambda ,
\end{equation*}%
(resp. 
\begin{equation*}
p(xy-x)\longrightarrow 0,
\end{equation*}%
when 
\begin{equation*}
y\in \{(a_{\lambda })\setminus \mathbf{1}\},\text{ }\lambda \in \Lambda 
\text{)},
\end{equation*}%
but 
\begin{equation*}
p(yx-x)\neq 0,
\end{equation*}%
(resp. 
\begin{equation*}
p(xy-x)\neq 0\text{)},
\end{equation*}%
due to the fact that 
\begin{equation*}
\mathbf{1}\notin \{(a_{\lambda })\setminus \mathbf{1}\}.
\end{equation*}
\end{proof}

Let us now turn again our attention to Jordan algebras. The next result is a
Jordan-algebraic analogue of Fritzsche's theorem from \cite{Fritzsche82} for
locally JB-algebras.

\begin{theorem}
Let $(J,\tau )$ be a real unital locally JB-algebra, and $x\in J$, be such
that $x\notin J_{b},$ where 
\begin{equation*}
J_{b}=\{x\in J:\left\Vert x\right\Vert _{\infty }=\underset{p\in \Gamma (E)}{%
\sup }p(x)<\infty \}.
\end{equation*}%
Then $(A,\tau )$ has a proper dense quadratic ideal $I$.
\end{theorem}

\begin{proof}
Let $(J,\tau )$ be a locally JB-algebra. Let us conside $J$ as a dense
quadratic ideal of itself. From Theorem 4 above it follows that $(J,\tau )$
has an qai $(a_{\lambda }),$ $\lambda \in \Lambda ,$ consisting of elements
of $J,$ such that: 

1). The net $(a_{\lambda }),$ $\lambda \in \Lambda ,$ is increasing, in the
sense that 
\begin{equation*}
a_{\lambda }\geq \mathbf{0},
\end{equation*}%
for every $\lambda \in \Lambda ,$ and 
\begin{equation*}
a_{\lambda }\leq a_{\nu },
\end{equation*}%
for any $\lambda \leq \nu $ in $\Lambda .$

2). 
\begin{equation*}
p(a_{\lambda })\leq 1,
\end{equation*}%
for all $p\in \Gamma ,$ $\lambda \in \Lambda .$

Let now 
\begin{equation*}
I=\{x:x\in J,\text{ }x=U_{y}z,\text{ where }y\in \{(a_{\lambda })\setminus 
\mathbf{1}\},\text{ }\lambda \in \Lambda ,\text{ and }z\in J\}.
\end{equation*}

Due to MacDonald Identity one can see that $I$ is a quadratic ideal of $J$.
One can easily see that $I$ is dense in $(J,\tau )$ because $(a_{\lambda }),$
$\lambda \in \Lambda ,$ is an qai of $(J,\tau ).$ On the other hand, one can
see that $I$ is a proper quadratic ideal of $(J,\tau ),$ because because for
the unbounded element $x\in A,$ 
\begin{equation*}
p(U_{y}x-x)\longrightarrow 0,
\end{equation*}%
when 
\begin{equation*}
y\in \{(a_{\lambda })\setminus \mathbf{1}\},\text{ }\lambda \in \Lambda ,
\end{equation*}%
but 
\begin{equation*}
p(U_{y}x-x)\neq 0,
\end{equation*}%
due to the fact that 
\begin{equation*}
\mathbf{1}\notin \{(a_{\lambda })\setminus \mathbf{1}\}.
\end{equation*}
\end{proof}


\begin{thebibliography}{99}
\bibitem{AlfsenSchultzStoermer78} \textbf{Alfsen, E.M.; Shultz, F.W.; St\o %
rmer, E.,} \textit{A Gelfand-Naimark theorem for Jordan algebras.} (English)
Advances in Math. Vol. 28 (1978), No. 1, pp. 11-56.

\bibitem{Apostol71} \textbf{Apostol, C.,} \textit{b*-algebras and their
representation.} (English) J. London Math. Soc. (2) No. 3 (1971), pp. 30--38.

\bibitem{Arens52} \textbf{Arens, R.,} \textit{A generalization of normed
rings.} (English), Pac. J. Math., Vol. 2 (1952), pp. 455-471.

\bibitem{BhattKaria93} \textbf{Bhatt, S.J.; Karia, D.J.,} \textit{An
intrinsic characterization of pro-C*-algebras and its applications.}
(English summary) J. Math. Anal. Appl. Vol. 175 (1993), No. 1, pp. 68--80.

\bibitem{Fragoulopoulou05} \textbf{Fragoulopoulou, M.,} \textit{Topological
algebras with involution.} (English) North-Holland Mathematics Studies, Vol.
200. Elsevier Science B.V., Amsterdam, 495 pp., (2005).

\bibitem{Fritzsche82} \textbf{Fritzsche, M.,} \textit{On the existence of
dense ideals in LMC*-algebras.} (German, Russian summary) Z. Anal.
Anwendungen Vol. 1 (1982), No. 3, pp. 81--84.

\bibitem{GelfandNaimark43} \textbf{Gelfand, I.M.; Naimark, M.A.,} \textit{On
the embedding of normed rings into the ring of operators in Hilbert space.}
(English. Russian summary) Rec. Math. [Mat. Sbornik] N.S. Vol. 12(54)
(1943), pp. 197-213.

\bibitem{Hanche-OlsenStoermer84} \textbf{Hanche-Olsen, H.; St\o rmer, E.,} 
\textit{Jordan operator algebras.} (English), Monographs and Studies in
Mathematics, Vol. 21. Boston - London - Melbourne: Pitman Advanced
Publishing Program. VIII, 183 pp., (1984).

\bibitem{Inoue71} \textbf{Inoue, A.,} \textit{Locally C*-algebras.}
(English), Mem. Fac. Sci. Kyushu Univ. (Ser. A), No. 25 (1971), pp. 197-235.

\bibitem{KatzFriedman06} \textbf{Katz, A.A.; Friedman, O.,} \textit{On
projective limits of real C*- and Jordan operator algebras.} (English),
Vladikavkaz Mathematical Journal, Vol. 8 (2006), No. 2, pp. 33-38.

\bibitem{Li03} \textbf{Li, B.,} \textit{Real operator algebras.} (English)
World Scientific Publishing Co., Inc., River Edge, NJ, 241 pp., (2003).

\bibitem{Mallios86} \textbf{Mallios, A.,} \textit{Topological algebras.
Selected topics.} (English) North-Holland Mathematics Studies, Vol. 124.
Notas de Matem\'{a}tica [Mathematical Notes], No. 109. North-Holland
Publishing Co., Amsterdam, 535 pp., (1986).

\bibitem{Michael52} \textbf{Michael, E.A.,} \textit{Locally
multiplicatively-convex topological algebras.} (English), Mem. Am. Math.
Soc., Vol. 11 (1952), 79 pp.

\bibitem{Pedersen79} \textbf{Pedersen, G.K.,} \textit{C*-algebras and their
automorphism groups.} (English), London Mathematical Society Monographs.
Vol. 14.London - New York -San Francisco: Academic Press., 416 pp., (1979).

\bibitem{Phillips88} \textbf{Phillips, N. C.,} \textit{Inverse limits of
C*-algebras.} (English) J. Operator Theory Vol. 19 (1988), No. 1, pp.
159--195.

\bibitem{Schmuedgen75} \textbf{Schm\"{u}dgen, K.,} \textit{\"{U}ber
LMC-Algebren.} (German) Math. Nachr.Vol. 68 (1975), pp. 167--182.
\end{thebibliography}
\end{document}